\documentclass[11pt,a4paper, twoside]{amsart}
\usepackage{url}
\usepackage{amssymb}
\usepackage{color}
\usepackage{fancyhdr}
\usepackage{colonequals}
\usepackage{graphicx}
\usepackage{pinlabel}

\input{epsf.tex}

\renewcommand{\epsilon}{\varepsilon}

\newtheorem*{namedtheorem}{\theoremname}
\newcommand{\theoremname}{testing}

\newtheorem{theorem}{Theorem}[section]

\newtheorem{corollary}[theorem]{Corollary}

\theoremstyle{definition}

\theoremstyle{remark}

\newcommand{\GL}{\operatorname{GL}}

\newcommand{\cohom}[3]{H^{{\raise1pt\hbox{$\scriptstyle#1$}}}(#2\>\!,#3)}
\newcommand{\tatecohom}[3]%
  {\widehat H^{{\raise1pt\hbox{$\scriptstyle#1$}}}(#2\>\!,#3)}

\newcommand{\Cohom}[3]%
  {H^{{\raise1pt\hbox{$\scriptstyle#1$}}}\big(#2\>\!,#3\big)}
\newcommand{\Tatecohom}[3]%
  {\widehat H^{{\raise1pt\hbox{$\scriptstyle#1$}}}\big(#2\>\!,#3\big)}

\newcommand{\homol}[3]{H_{{\lower1pt\hbox{$\scriptstyle#1$}}}(#2\>\!,#3)}
\newcommand{\homolog}[2]{H_{{\lower1pt\hbox{$\scriptstyle#1$}}}(#2)}

\DeclareMathOperator{\Out}{Out}
\DeclareMathOperator{\Mod}{Mod}

\DeclareMathOperator{\Aut}{Aut}

\newcommand{\lk}{\operatorname{lk}}
\newcommand{\st}{\operatorname{st}}

\begin{document}

\title{A note on nilpotent subgroups of automorphism groups of RAAGs}

\author{Javier Aramayona \& Anthony Genevois}

\begin{abstract}
We observe that automorphism groups of right-angled Artin groups contain nilpotent non-abelian subgroups, namely $H_3(\mathbb{Z})$ the three-dimensional integer Heisenberg group, provided they admit a certain type of element, called an adjacent transvection. This represents a (minor) extension of  a result of Charney-Vogtmann. 
\end{abstract}

\maketitle

\section{Introduction}

Automorphism groups of right-angled Artin groups (RAAGs) are normally studied through their comparison with linear groups and automorphism groups of free groups (and thus, by extension, with mapping class groups). One fact that distinguishes the linear group $\GL(n,\mathbb{Z})$ from the automorphism group $\Aut(\mathbb{F}_n)$ of the free group $\mathbb{F}_n$, and from  the mapping class group  $\Mod(S)$,  is that every solvable subgroup of the latter two is virtually abelian; see \cite{bestvina,emina} and \cite{BLM}, respectively. In sharp contrast, arbitrary automorphism groups of RAAGs may contain a copy of $\GL(n,\mathbb{Z})$ with $n\ge 3$. Even when this is not the case, Charney-Vogtmann \cite{CV} proved that (outer) automorphism groups of RAAGs contain torsion-free nilpotent and non-abelian  subgroups, whenever they contain at least two {\em adjacent transvections}; see below for a definition. Further examples  may be found in \cite{Day?}. 

The purpose of this short note is to observe that the presence of a single adjacent transvection suffices for the (full) automorphism group of a centerless RAAG to contain a nilpotent non-abelian subgroup. We remark that Charney-Vogtmann showed \cite{CV}, if there are no adjacent transvections then every solvable subgroup is virtually abelian.

Before we state our result, denote by $\le$ the usual partial order in the vertex set $V(\Gamma)$ of $\Gamma$; see Section \ref{sec:defs}. Also, let  $H_3(\mathbb Z)$ be the three-dimensional integer Heisenberg group, which has presentation $$H_3(\mathbb Z) = \langle A,B,C \mid  [A,C] = [B,C] = 1, [A, B] = C \rangle. $$  With this notation, we will observe:

\begin{theorem} Let $\Gamma$ be a simplicial graph. If there are adjacent vertices $a,b \in V(\Gamma)$ with $a \le b$, such that $b$ is not adjacent to all the vertices of $\Gamma$, then $H_3(\mathbb Z)$ is a subgroup of $\Aut(A_\Gamma)$. 
\label{thm:heisenberg-aut}
\end{theorem}

As a consequence, for such $\Gamma$ the group $\Aut(A_\Gamma)$ does not embed in $\Mod(S)$ or $\Aut(\mathbb{F}_n)$. A further immediate application, surely well-known to experts, is that no finite-index subgroup of such $\Aut(A_\Gamma)$ may act nicely on non-positively curved space:

\begin{corollary}
If $\Gamma$ is as in Theorem \ref{thm:heisenberg-aut} , then no finite-index subgroup of $\Aut(A_\Gamma)$ can act properly by semi-simple isometries on a ${\rm CAT}(0)$ space. 
\label{cor:cat}
\end{corollary}


In sharp contrast, $\Aut(A_\Gamma)$ and $\Out(A_\Gamma)$ are sometimes commensurable to a right-angled Artin group \cite{CF,Day?}, and therefore they act nicely on ${\rm CAT}(0)$ spaces.

\section{Definitions}
\label{sec:defs}

In order to make this note as concise and self-contained as possible, we only introduce the objects that we will need in the proof of our results. We refer the reader to the various papers in the bibliography below for a thorough introduction to automorphism groups of right-angled Artin groups.  

Let $\Gamma$ be a simplicial graph, and denote by $V(\Gamma)$ (resp. $E(\Gamma)$) its set of  vertices (resp. edges). The right-angled Artin group $A_\Gamma$ defined by $\Gamma$ is the group given by the presentation $$A_\Gamma= \left \langle v \in V(\Gamma) \mid [v,w] = 1 \iff vw \in E(\Gamma) \right \rangle.$$ 

Given a right-angled Artin group $A_\Gamma$, we consider its automorphism group $\Aut(A_\Gamma)$. This is a finitely presented group with an explicit generating set \cite{Laurence,Servatius} and an explicit (although in slightly different terms) presentation \cite{Day}. 

Here, we will need a specific type of element of $\Aut(A_\Gamma)$ and $\Out(A_\Gamma)$, called a {\em tranvection}. Given vertices $v, w \in V(\Gamma)$, the transvection $t_{vw}$ is the self-map of $A_\Gamma$ defined by  $$t_{vw}: v \mapsto vw$$ and $t_{vw}(u)=u$ for every $u \ne v$. An easy observation is that $t_{vw} \in \Aut(A_\Gamma)$ if and only if $\lk(v) \subset \st(w)$; here, $\lk(\cdot)$ denotes the link of a vertex in $\Gamma$, while $\st(\cdot)$ denotes its star, namely the link union the vertex. As usual in the literature, we will write $v \le w$ to mean $\lk(v) \subset \st(w)$, noting that the relation $\le$ is in fact a partial order.  Finally, we say that $t_{vw}$ is an {\em adjacent transvection} if $v$ and $w$ are adjacent in $\Gamma$.  

\section{A homomorphic image of the Heisenberg group} 

We now prove Theorem \ref{thm:heisenberg-aut}:

\begin{proof}[Proof of Theorem \ref{thm:heisenberg-aut}]
Let $a, b$ be adjacent vertices of $\Gamma$, with $a \le b$ and which are not adjacent to all the vertices of $\Gamma$. We write $c_a$ and $c_b$ for the automorphisms of $A_\Gamma$ given by  conjugation by $a$ and $b$, respectively. Finally, let $t= t_{ab}$ the transvection that sends $a$ to $ab$ and fixes the rest of generators. 

First, observe that $c_a$ and $c_b$ commute since $a$ and $b$ are adjacent. Next, we claim that $[c_b, t]=1$ also. Indeed, if $v \ne a$, $$tc_bt^{-1}( v) = tc_b(v) = t(bvb^{-1})= bvb^{-1} = c_b(v),$$ and 
$$tc_bt^{-1}(a) = tc_b(ab^{-1})= t(bab^{-2}) = babb^{-2} = c_b(a).$$ Finally, we claim that $tc_at^{-1} = c_bc_a$. Indeed, if $v\ne a$ then $$tc_at^{-1}(v) = tc_a(v)= t(ava^{-1})= ab v b^{-1} a^{-1} = c_bc_a(v),$$ while $$tc_at^{-1}(a) = tc_a(ab^{-1}) = t(a^2b^{-1}a^{-1}) = ababb^{-1}b^{-1}a^{-1} = c_bc_a(a),$$ as desired. 

In light of the above, and using notation for the presentation of $H_3(\mathbb Z)$ given in the introduction, the map $H_3(\mathbb Z) \to \Aut(A_\Gamma)$ given by $A \mapsto c_a$, $B\mapsto t$, and $C \mapsto c_b$ is a homomorphism. Now, we want to prove that this is a monomorphism.

From the presentation of $H_3(\mathbb{Z})$, it clearly follows that any of its elements can be written as $A^mB^nC^p$ for some $m,n,p \in \mathbb{Z}$. Furthermore, $A^mB^nC^p=1$ if and only if $m=n=p=0$. Indeed, we first deduce from the equality $A^mB^nC^p=1$ that $m=n=0$ by looking at the image of $A^mB^nC^p$ into the abelianization of $H_3(\mathbb{Z})$; therefore, our equality reduces to $C^p=1$, and finally $p=0$ follows from the torsion-freeness of $H_3(\mathbb{Z})$. Thus, in order to prove that our homomorphism $H_3(\mathbb Z) \to \Aut(\Gamma \mathcal{G})$ is injective, we only have to verify that, for every $m,n,p \in \mathbb{Z}$, $c_a^mt^nc_b^p=1$ implies $m=n=p=0$.

So let $m,n,p \in \mathbb{Z}$ and suppose that $c_a^mt^nc_b^p=1$. Noticing that $$ c_a^mt^nc_b^p(a)=c_a^mt^n(b^pab^{-p})= c_a^mt^n(a)=c_a^m(ab^n)=a^mab^na^{-m}=ab^n,$$ we first deduce that $n=0$. Therefore, $c_a^mc_b^p=1$. This precisely means that $a^mb^p$ belongs to the center of $A_\Gamma$. On the other, the center of $A_\Gamma$ corresponds exactly to the subgroup generated by the vertices which are adjacent to all the vertices of $\Gamma$. Because we supposed that neither $a$ nor $b$ is adjacent to all the vertices of $\Gamma$, it follows that $a^mb^p=1$. Therefore, we deduce that $a^m=b^p=1$, and finally that $m=p=0$ by torsion-freeness of $A_\Gamma$. 
\end{proof}

\section{Actions on ${\rm CAT}(0)$ spaces}

We prove Corollary \ref{cor:cat}: 

\begin{proof}
Let $\Gamma$ as in Theorem \ref{thm:heisenberg-aut}, so there exist adjacent vertices $a,b \in V(\Gamma)$ with $a \le b$ and which are not adjacent to all the vertices of $\Gamma$.
As above, write $c_a$ and $c_b$ for the conjugations on $a$ and $b$, respectively, and $t$ for the transvection $t_{ab}$. Notice that $c_a$ and $c_b$ are infinite-order automorphisms because $a$ and $b$ do not belong to the center of $A_\Gamma$. 

Consider a finite-index subgroup $G < \Aut(A_\Gamma)$; as such, there exists some $m \geq 1$ such that $c_a^m, c_b^m, t^m \in G$. It is immediate to check that, for all $n \in \mathbb{N}$, 
\begin{equation}
t^n c_a^m t^{-n} = c_a^{m} c_b^{mn}.
\label{eq}
\end{equation} Suppose now that $G$ acts properly by semi-simple isometries on some ${\rm CAT}(0)$ space $X$. As $c_a^m$ and $c_b^m$ commute, the Flat Torus Theorem \cite{BH} implies that $X$ contains an isometrically embedded copy of a Euclidean plane,  on which  $c_a^m$ and $c_b^m$ act by translations with quotient a $2$-torus. It follows that, as a transformation of this plane, the translation length of  $c_a^{m} c_b^{mn}$ must tend to infinity as $n$ grows; on the other hand, equation (\ref{eq}) implies that this translation length must be equal to that of $c_a^m$. This is a contradiction, and the result follows. 
\end{proof}

A possible interpretation of the previous proof is the following. As Theorem \ref{thm:heisenberg-aut} proves, $\Aut(\Gamma \mathcal{G})$ contains a copy of the three-dimensional integer Heisenberg group $$H_3(\mathbb{Z}) = \langle A,B,C \mid [A,C]=[B,C]=1, [A,B]=C \rangle.$$ It is not difficult to notice that, for any $k \geq 1$, the subgroup $\langle A^k,B^k,C^k \rangle \leq H_3(\mathbb{Z})$ is naturally isomorphic to $H_3(\mathbb{Z})$ itself, so that any finite-index subgroup of $\Aut(\Gamma \mathcal{G})$ has to contain a copy of $H_3(\mathbb{Z})$. Therefore, Corollary \ref{cor:cat} follows from the fact that $H_3(\mathbb{Z})$ does not act properly by simple-isometries on a ${\rm CAT}(0)$ space. This is essentially what we have shown in the previous proof. An alternative argument can be found in \cite[Corollary 5.1]{FSY}, where it is furthermore proved that, for any proper action of $H_3(\mathbb{Z})$ on a ${\rm CAT}(0)$ space, $C$ is necessarily parabolic. It is worth noticing that $H_3(\mathbb{Z})$ has a proper parabolic action on the complex hyperbolic plance $\mathbb{H}_{\mathbb{C}}^2$, which is a proper finite-dimensional ${\rm CAT}(-1)$ space \cite[Corollary 5.1]{FSY}.

\end{document}